\documentclass[10pt]{amsart}
\newtheorem{theorem}{Theorem}[section]
\newtheorem{lemma}[theorem]{Lemma}

\theoremstyle{definition}
\newtheorem{definition}[theorem]{Definition}

\theoremstyle{remark}
\newtheorem{remark}[theorem]{Remark}

\def\innprod#1#2{\left<#1,#2\right>}
\numberwithin{equation}{section}
\begin{document}
\title{SHARP UPPER BOUND FOR THE FIRST EIGENVALUE}
\author{Binoy}
\address{Department of Mathematics and Statistics, Indian Institute of Technology, Kanpur 208016, India}
\email{binoy@iitk.ac.in}
\author{G. Santhanam}
\address{Department of Mathematics and Statistics, Indian Institute of Technology, Kanpur 208016, India}
\email{santhana@iitk.ac.in}
\thanks{First author is supported by a research fellowship of CSIR, India.}


\begin{abstract} 
 Let $M$ be a closed hypersurface in a noncompact rank-1 symmetric space $(\overline{\mathbb{M}}, ds^2)$ with $-4 \leq K_{\overline{\mathbb{M}}}
\leq -1,$ or in a complete, simply connected Riemannian manifold $\mathbb{M}$ such that $0 \leq K_{\mathbb{M}} \leq \delta^2$ or
$K_{\mathbb{M}} \leq k$ where $k = -\delta^2$ or $0$. In this paper we give sharp upperbounds for the first eigenvalue of laplacian of $M$.
\end{abstract}

\maketitle
\section{Introduction}\label{intro}
Let $(\overline{M}, g)$ be a complete Riemannian manifold of dimension $n \geq 2,$ and $M$ be a closed hypersurface. 
Starting with the work of Bleecker-Weiner \cite{bw}, there have been several works which give sharp upper bound for
the first eigen value $\lambda_1(M)$ of laplacian on $M$. In \cite{bw}, they proved that for  a hypersurface $M$ of $\mathbb{R}^{n+1}$, the first eigenvalue 
of $M$ is bounded above by 
$ \frac{1}{\text{Vol}(M)}\int_M |A|^2$, where $|A|^2$ is the square of the length of the second fundamental form of $M$. Reilly \cite{r} 
proved that if $M$ is a compact $n$-dimensional manifold which is isometrically immersed in $\mathbb{R}^{n+p} $, then $\lambda_1(M) \leq 
 \frac{n}{\text{Vol}(M)}\int_M |H|^2,$ where $H$ is the mean curvature vector.
This result was later extended in various ways to submanifolds of simply connected space forms (\cite{gjf}, \cite{eh}). 
If $M$ is hypersruface in a rank-one symmetric spaces, it was proved in \cite{gs} that $ \lambda_1(M) \leq \frac{1}{\text{Vol}(M)}\int_M \lambda_1(S(r)), $ where $\lambda_1(S(r)) $ is the first eigen 
value of the geodesic sphere in the ambient space with center at the center of mass of $M$ corresponding to the mass distribution function $\frac{1}{t}$.
All the above inequalities are sharp as equality holds if and only if the hypersurface $M$ is a geodesic sphere. 

For a closed hypersruface $M$ which is contained in a ball of radius less than $\frac{i(\mathbb{M}(k))}{4},$ and 
bounding a convex domain $\Omega$ in the simply connected space form $\mathbb{M}(k),\, k= $ 0 or 1, the second author proved \cite{gs1} that
$$
\frac{\lambda_1(M)}{\lambda_1(S(R))} \leq \frac{Vol(M)}{Vol(S(R))}
$$
where $R$ is such that $Vol(B(R)) = Vol(\Omega)$. A similar result was also obtained for $k = -1$.
Furthermore equality holds if and only if $M$ is a geodesic sphere of radius $R$.

In this paper we extend the results in \cite{gs1} to a wider class of Riemannian manifolds.
We denote by $K_M$, the sectional curvature of a Riemannian manifold $M$. 
We consider noncompact rank-1 symmetric space $(\overline{\mathbb{M}}, ds^2)$ with the metric $ds^2$ 
such that $-4 \leq K_{\overline{\mathbb{M}}} \leq -1$ or complete,
simply connected Riemannian manifold $\mathbb{M}$ such that $0 \leq K_{\mathbb{M}} \leq \delta^2,$ or $K_{\mathbb{M}} \leq k,$ where 
$ k = -\delta^2$ or $0$. 
Let $M$ be a closed hypersurface in $\mathbb{M}$ or $\overline{\mathbb{M}}$. In the case of $0 \leq K_{\mathbb{M}} \leq \delta^2$, we prove the following 
isoperimetric upper bound   
$$
\frac{\lambda_1(M)}{\lambda_1(S_{\delta}(R))} \leq \frac{Vol(M)}{Vol(S_{\delta}(R))} 
$$
where $S_{\delta}(R)$ is the  geodesic sphere of radius $R$ in the constant curvature space $\mathbb{M}(\delta^2)$(See theorem \ref{thm1} for the statement). 
We obtain similar isoperimetric upperbounds in the other cases also (see theorem \ref{thm2} and \ref{thm3} for statements).
These upper bounds are sharp and equality holds if and only if the hypersurface is a geodesic sphere. 

We refer to \cite{ic} and \cite{dc} for the basic Riemannian geometry used in this paper.
\section{Statement of Results}\label{state}
To state the results we need the notion of center of mass of a subset of a Riemannian manifold.   	 

Let $(\overline{M},g)$ be a $(n+1)$ dimensional complete Riemannian manifold. For a point $p \in \overline{M}$, we denote by $c(p)$ the
convexity radius of $(\overline{M},g)$ at $p$.
For a subset $A \subset B(q,c(q)) $, for $q \in \overline{M}$, we let $CA$ denote the convex hull of $A$. Let $exp_q:T_q\overline{M} \rightarrow \overline{M}$ 
be the exponential map and
$ X = (x_1, x_2, ... , x_{n+1})$ be the normal coordinate system at $q$. We identify $CA$ with $exp_q^{-1}(CA)$ and denote $g_q(X,X)$ as 
$\parallel \! X \! \parallel_q^2$ for $X \in T_q\overline{M}$. We state the center of mass theorem below.
\begin{theorem}
Let $A$ be a measurable subset of $(\overline{M},g)$ contained in $B(q_0, c(q_0))$ for some point $q_0 \in \overline{M}$. 
Let $G: [0, 2c(q_0)] \rightarrow \mathbb{R}$ be a 
continuous function such that $G$ is positive on $(0, 2c(q_0))$. Then there exists a point $p \in CA$ such that 
$$
\int_AG(\parallel \!X \!\parallel_p)XdV = 0,
$$
where $X = (x_1, x_2, ... , x_{n+1})$ is a geodesic normal coordinate system at $p$.
\end{theorem}
For a proof see \cite{gs} or \cite{eh}.
\begin{definition}
 The point $p$ in the above theorem is called as a center of mass of the measurable subset $A$ with respect to the mass distribution function $G$.
\end{definition}

Before stating the results, we fix some notations which will be used throughout the paper.
Let $(\overline{\mathbb{M}}, ds^2)$ be a noncompact rank-1 symmetric space with the metric $ds^2$ such that the sectional curvature
satisfies $-4 \leq K_{\overline{\mathbb{M}}} \leq -1$. Let the dimension of $(\overline{\mathbb{M}}, ds^2)$ be  $kn,$ where $k \, 
 = \,  dim_{\mathbb{R}}\mathbb{K}; \,\mathbb{K} = \mathbb{R}, \mathbb{C}, \mathbb{H}$ or $\mathbb{C}a$. Fix a point $p \in \overline{\mathbb{M}}$
and let $\gamma$ be a geodesic starting at $p$. Then the volume density function along $\gamma$ at the point 
$\gamma(r)$ is given by $sinh^{kn-1}r\,cosh^{k-1}r.$ Also we denote by $S(r),$ the geodesic sphere of radius $r$ with center $p$ and by
$\Delta_{S(r)},$ the Laplacian of $S(r)$. Let $\lambda_1(S(r))$ be the first eigenvalue of $\Delta_{S(r)}$. It is well known (\cite{eh}, \cite{gs})
that for $r > 0$
$$
\lambda_1(S(r))= \frac{kn-1}{\sinh^2r} - \frac{k-1}{cosh^2r}.
$$  

For a given $\delta > 0,$ let $\mathbb{M}$ denote a complete, simply connected Riemannian manifold of dimension $(n+1)$ such that the sectional 
curvature satisfies one of the following conditions: 
\begin{enumerate}
\item 
$0 \leq K_{\mathbb{M}} \leq \delta^2$ 
\item 
$K_{\mathbb{M}} \leq 0 $
\item 
$K_{\mathbb{M}} \leq  -\delta^2.$
\end{enumerate} 
Also we denote by $\mathbb{M}(k),$ the simply connected space 
form of dimension $(n+1)$ with constant curvature $ k = \pm \delta^2$ or $0.$ 
For $r\geq 0$, let  
$$
\sin_{\delta}r = 
\begin{cases}
 \frac{1}{\delta}\sin\delta \,r & if \, \, 0 \leq K \leq \delta^2 \\
 r & if \, K \leq 0 \\ 
 \frac{1}{\delta}\sinh\delta \,r & if \, K \leq -\delta^2   
\end{cases} \ \ 
\text{and} \ \ 
\cos_{\delta}r = 
\begin{cases}
 \cos\delta \,r & if \, \, 0 \leq K \leq \delta^2 \\
 1 & if \,  K \leq 0 \\ 
 \cosh\delta \,r & if \, K \leq -\delta^2.  
\end{cases}
$$

Let $M$ be a closed hypersurface in $\mathbb{M} $ or in $ \overline{\mathbb{M}}$ and $\Omega$ be a bounded domain whose boundary is $M$.
In the case of $\mathbb{M}$ with $0 \leq K_{\mathbb{M}} \leq \delta^2$, we always assume that $M$ is contained in a ball of radius less 
than min$(\frac{\pi}{4\delta}, inj(\mathbb{M}))$.  
Let $p \in C\Omega$ be a center of mass corresponding to the function $\frac{\sin_{\delta}r}{r}.$
The geodesic polar coordinate system centered at $p$ is denoted by $(r,u)$ where $r> 0$ and $u \in U_p\mathbb{M}(\text{or}\, U_p{\overline{\mathbb{M}}}).$
For any $q \in M$, let $\gamma_q$ be the unique unit speed geodesic segment joining $p$ and $q$ with $ \gamma_q^{\prime}(0) = u.$ 
We write $d(p,q)$ as $t_q(u)$. Consider $W \subset T_p\mathbb{M}$ such that $\Omega = exp_p(W)$. Fix a point $p_0 \in \mathbb{M}(k)$ 
and an isometry $i:T_p\mathbb{M} \rightarrow T_{p_0}\mathbb{M}(k)$. Let $\Omega_{\delta} = exp_{p_0}(i(W)),M_{\delta} = 
\partial \Omega_{\delta}$ and for $\bar{q} \in M_{\delta}$ we write $d(p_0, \bar{q}) = t_{\bar{q}}(\bar{u})$ where $\bar{u}$ is the tangent 
at $p_0$ of the unit speed geodesic segment $\gamma_{\bar{q}}$ joining between $p_0 $ and $\bar{q}$. We also denote by $\phi$ and $\phi_{\delta},$
the volume density functions of $\mathbb{M}$ and $\mathbb{M}(k)$ along the radial geodesics starting at $p$ and $p_0$ respectively.

With these notations we state the main results. 

\begin{theorem}\label{thm1}
Let $\mathbb{M}$ be a complete, simply connected $(n+1)$ dimensional manifold such that $0 \leq K_{\mathbb{M}} \leq \delta^2$ or $K_{\mathbb{M}} \leq 0$.
Let $M$ be a closed hypersurface in $\mathbb{M}$ which encloses a bounded region $\Omega$. Then 
$$
\frac{\lambda_1(M)}{\lambda_1(S_{\delta}(R))} \leq \frac{Vol(M)}{Vol(S_{\delta}(R))} 
$$
where $R > 0$ is such that $Vol(\Omega_{\delta}) = Vol(B_{\delta}(R))$; here $B_{\delta}(R)$ and $S_{\delta}(R)$ are the geodesic ball and geodesic sphere 
respectively of radius $R$ in the constant curvature space $\mathbb{M}(k),$ where $ k = \delta^2$ or $0$.

Further, the equality holds if and only if $M$ is a geodesic sphere in $\mathbb{M}$ and $\Omega$ is isometric to $B_{\delta}(R)$. 
\end{theorem}
For the case $ K \leq -\delta^2,$ we have  
\begin{theorem}\label{thm2}
Let $\mathbb{M}$ be a complete, simply connected $(n+1)$ dimensional manifold such that $ K \leq -\delta^2$. Let $M$ be a closed
hypersurface in $\mathbb{M}$ which encloses the bounded region $\Omega$. Then 
$$
\frac{\lambda_1(M)}{\lambda_1(S_{\delta}(R))} \leq \frac{Vol(M)}{Vol(S_{\delta}(R))}\ +
  \frac{1}{n\,Vol(S_{\delta}(R))}\int_M \parallel \nabla^M\sin_{\delta}r \parallel^2 $$
where $R>0$ is such that $Vol(\Omega_{\delta}) = Vol(B_{\delta}(R))$; here $B_{\delta}(R)$ and $S_{\delta}(R)$ are the geodesic ball and 
geodesic sphere respectively of radius $R$ in the constant curvature space $\mathbb{M}(-\delta^2)$. 

Further,the equality holds if and only if $M$ is a geodesic sphere and $\Omega$ is isometric to $B_{\delta}(R)$.
\end{theorem}
In the case of rank-1 symmetric space of noncompact type, we have 
\begin{theorem}\label{thm3}
Let $(\overline{M}, ds^2)$ be a non compact rank-1 symmetric space with $\text{dim}\overline{M} = kn $ where $k
= \text{dim}_{\mathbb{R}}\mathbb{K}; \, \mathbb{K} = \mathbb{R}, \mathbb{C}, \mathbb{H}$ or $\mathbb{C}a$.
Let $M$ be a closed hypersurface in $\overline{M}$ which encloses the bounded region $\Omega$. Then for $k=1$ we have
\vspace{-.3cm}
\begin{eqnarray*}
\frac{\lambda_1(M)}{\lambda_1(S(R))} & \leq &  \frac{Vol(M)}{Vol(S(R))} + \frac{1}{(n-1)Vol(S(R))}\int_M \parallel \nabla^M\sinh\,r \parallel^2
\end{eqnarray*}
and for $k>1$ we have
\begin{eqnarray*}
\lambda_1(M) & \leq & \lambda_1(S(R))\left(\frac{Vol(M)}{Vol(S(R))}\right) + \frac{k-1}{\cosh^2R}\left(\frac{Vol(M)}{Vol(S(R))}\right) \\
& & + \ \frac{1}{\sinh^2R\,Vol(S(R))}\!\int_M\! \parallel \!\nabla^M\sinh\,r \!\parallel^2
\end{eqnarray*}
where $R>0$ is such that $Vol(\Omega) = Vol(B(R))$; here $B(R)$ and $S(R)$ are the geodesic ball and geodesic sphere respectively of radius $R$.
Further, the equality holds in above two inequalities if and only if $M$ is a geodesic sphere of radius $R$.  
\end{theorem}
\section{Preliminaries}\label{prelm}
Let $M$ be a closed hypersurface in $\mathbb{M} $ or in $ \overline{\mathbb{M}}$ and let $\Omega$ be the bounded domain whose boundary is $M$.
Fix a point $p \in \Omega$. Then for every point $q \in M,$ there exist unique geodesic segment $\gamma_q$ such that 
$\gamma_q(0) = p, \gamma_q^{\prime}(0) = u$ and $\gamma_q(t_q(u)) = q$. We observe that this geodesic segment may intersect $M$ at points other than $q$. 
For $u \in U_p\mathbb{M} \, (\text{or} \,U_p\overline{\mathbb{M}})$, let 
\begin{equation*}
r(u) = max\{r>0 \, | \, exp_p(ru) \in M\}
\end{equation*} 
and define 
$$
A = \{exp_p(r(u)u) \, | \, u \in U_p\mathbb{M} \, (\text{or} \,U_p\overline{\mathbb{M}})\}.
$$
Then $A \subset M$ and hence for any non-negative measurable function $f$ on $M$, we have $\int_Mf \geq \int_A f$. 

We now prove the following lemma which is crucial in the proofs of main results. 
\begin{lemma}\label{lm1} 
Let $M$ be a closed hypersurface in $\mathbb{M}$ or in $\overline{\mathbb{M}}$ and $\Omega$ be a bounded domain with boundary $ \partial \Omega = M.$
Fix a point $p \in \Omega$. Then the following holds:
\begin{enumerate}
 \item If $M \subseteq \mathbb{M},$ then
\begin{eqnarray*}
\int_M \sin_{\delta}^2\,d(p,q)dm \geq Vol(S_{\delta}(p_0, R))\sin_{\delta}^2\,R
\end{eqnarray*}
where $dm$ is the measure on $M, \, S_{\delta}(p_0, R)$ is the geodesic sphere and $ B_{\delta}(p_0, R)$ is the
geodesic ball of radius $R$, centered at $p_o$ in the space form 
$\mathbb{M}(k)$ and $R>0$ is such that $Vol(\Omega_{\delta}) = Vol(B_{\delta}(p_0, R))$.

Further, the equality holds if and only if $M$ is a geodesic sphere in $\mathbb{M}$ and $\Omega $ is isometric to $ B_{\delta}(p_0, R)$.
\item If $M \subseteq \overline{\mathbb{M}},$ then
\begin{eqnarray*}
\int_M \sinh^2d(p,q)dm \geq Vol(S(p, R))\sinh^2R
\end{eqnarray*}
where $dm$ is the measure on $M, \, S(p, R)$ is the geodesic sphere and $ B(p, R)$ is the geodesic ball of radius $R$ centered at $p$ in $\overline{\mathbb{M}}$ 
and $R>0$ is such that $Vol(\Omega) = Vol(B(p, R))$.

The equality holds if and only if $M$ is a geodesic sphere centered at $p$ of radius $R$.  
\end{enumerate}
\end{lemma}
\begin{proof}
We begin by considering $M \subset \mathbb{M}$. Let $q \in M$ and $\phi (t_q(u))$ be the volume density of the geodesic sphere $S(p, t_q(u))$ at the point $q$.
Let $\theta (q)$ be the angle between the 
unit outward normal $\eta (q)$ to $M$ and the radial vector $\partial r(q)$. Let $du$ be the spherical volume density of the unit sphere 
$U_p\mathbb{M}$. Then we know that (\cite{mtw}, p.385, or \cite{tf}, p.1097) $dm(q) = \sec\theta (q)\phi(t_q(u))du$. Hence, 
\begin{eqnarray*}
\int_M \sin_{\delta}^2\,d(p,q)dm(q) & \geq & \int_A \sin_{\delta}^2\,d(p,q)dm(q) \\
& = & \int_{U_p\mathbb{M}}\sin_{\delta}^2\,t_q(u) \, \sec\theta (q)\,\phi(t_q(u))du\\
 & \geq & \int_{U_p\mathbb{M}} \sin_{\delta}^2\,t_q(u) \,\phi(t_q(u))du.
\end{eqnarray*}
For $q \in M,$ consider the unit speed geodesic segment $\gamma_q$ in $\mathbb{M}$ joining $p$ and $q$ and the corresponding geodesic segment
$\gamma_{\bar{q}}$ joining $p_0 $ and $\bar{q} \in M_{\delta}$ in the space form $\mathbb{M}(k)$.
Then by Rauch comparison theorem \cite{dc} it follows that $l(\gamma_q) \geq l(\gamma_{\bar{q}})$ and 
hence $t_q(u) \geq t_{\bar{q}}(\bar{u})$. By Gunther's volume comparison theorem \cite{ghl} 
we also have $\phi (t_q(u)) \geq \phi_{\delta}(t_q(u)) = \sin_{\delta}^nt_q(u)$ along the geodesics $\gamma_p $ and $\gamma_{\bar{q}}$ respectively.
Hence, 
\begin{eqnarray}\nonumber
\int_M \sin_{\delta}^2\,d(p,q)dm(q) & \geq & \int_{U_p\mathbb{M}}\sin_{\delta}^2t_q(u) \,\phi(t_q(u))du \nonumber \\
& \geq & \int_{U_p\mathbb{M}}\sin_{\delta}^2t_q(u) \,\phi_{\delta}(t_q(u))du \nonumber \\
& \geq & \int_{U_p\mathbb{M}(k)} \sin_{\delta}^2t_{\bar{q}}(\bar{u}) \,\phi_{\delta}(t_{\bar{q}}(\bar{u}))d\bar{u} \nonumber\\
& = & \int_{U_p\mathbb{M}(k)}\sin_{\delta}^{n+2}\,t_{\bar{q}}(\bar{u})d\bar{u} \nonumber \\
& = & (n+2)\int_{U_p\mathbb{M}(k)} \int_0^{t_{\bar{q}}(\bar{u})} \sin_{\delta}^{n+1}\,r\,\cos_{\delta} \,r \, dr\,d\bar{u} \nonumber \\
& \geq & (n+2)\int_{\Omega_{\delta}} \sin_{\delta} \,r\,\cos_{\delta} \,r\,dV.\nonumber 
\end{eqnarray}
Let $R>0$ be such that $Vol(\Omega_{\delta}) = Vol(B_{\delta}(p_0,R))$ and $f(r) = \sin_{\delta}r \, \cos_{\delta}r$. 
We observe that 
\begin{equation*}
   f  \ \ \text{is increasing on} \ \   
\begin{cases}
 [0, \frac{\pi}{4\delta}) & if \, \, 0 \leq K \leq \delta^2 \\
 [0, \infty) & if \, K \leq 0
 \end{cases}  
\end{equation*}
and $ Vol(\Omega_{\delta}\backslash(\Omega_{\delta} \cap B_{\delta}(p_0,R)))  =  Vol(B_{\delta}(p_0,R)\backslash (\Omega_{\delta}\cap B_{\delta}(p_0,R))).$ 
Using these facts we get 
\begin{eqnarray*}
\int_{\Omega_{\delta}} f(r)\,dV & = & \int_{\Omega_{\delta}\cap B_{\delta}(p_0,R)} f(r)\,dV + 
\int_{\Omega_{\delta}\backslash(\Omega_{\delta} \cap B_{\delta}(p_0,R))} f(r)\,dV \\
& = & \int_{B_{\delta}(p_0,R)}f(r)\,dV - \int_{B_{\delta}(p_0,R)\backslash \Omega_{\delta}\cap B_{\delta}(p_0,R)} f(r)\,dV \\  
& & + \int_{\Omega_{\delta}\backslash(\Omega_{\delta} \cap B_{\delta}(p_0,R))} f(r)\,dV
\end{eqnarray*}
\begin{eqnarray*} 
& \geq & \int_{B_{\delta}(p_0,R)}f(r)\,dV - \int_{B_{\delta}(p_0,R)\backslash \Omega_{\delta}\cap B_{\delta}(p_0,R)} f(r)\,dV\\
& & +\int_{\Omega_{\delta}\backslash(\Omega_{\delta} \cap B_{\delta}(p_0,R))} f(R)\,dV \\
& = & \int_{B_{\delta}(p_0,R)}f(r)\,dV + 
\int_{B_{\delta}(p_0,R)\backslash \Omega_{\delta}\cap B_{\delta}(p_0,R)} (f(R)- f(r))\,dV \\
& \geq & \int_{B_{\delta}(p_0,R)}f(r)\,dV \\
& = & \int_{U_p\mathbb{M}(k)}\int_0^R f(r)\sin_{\delta}^n \,r\,dr\,d\bar{u} \\
& = & \int_{U_p\mathbb{M}(k)} \int_0^R \sin_{\delta}^{n+1} \,r\,\cos_{\delta}\,r\,dr\,d\bar{u} \\
& = & \frac{1}{n+2}\int_{U_p\mathbb{M}(k)} \sin_{\delta}^{n+2}\,R \, d\bar{u}\\ 
& = & Vol(S_{\delta}(R))\frac{\sin_{\delta}^2R}{n+2}. 
\end{eqnarray*}
Thus we get 
$$
\int_M \sin_{\delta}^2 \,d(p,q)dm(q) \geq Vol(S_{\delta}(R))\sin_{\delta}^2 \,R.
$$

Further the equality holds in the above inequality if and only if the following conditions hold:
\begin{enumerate}
\item 
$\sec \, \theta(q) = 1$ and $l(\gamma_q) = l(\gamma_{\bar{q}})$  for all points $q \in M.$
\item
$\phi (r)  = \phi_{\delta}(r)$ for $r \,  \leq \text{diam}(\mathbb{M})$ along the geodesics $\gamma_p $ and $\gamma_{\bar{q}}$ respectively. 
\item 
$Vol(B_{\delta}(p_0,R)\backslash (\Omega_{\delta}\cap B_{\delta}(p_0,R))) = 0$.
\end{enumerate}
Now $\sec \, \theta(q) = 1$ implies that the outward normal $\eta(q) = \partial r(q).$ Thus the first condition implies that $\eta(q) = \partial r(q)$ 
for all points in $q \in M$. This shows that $M$ is a 
geodesic sphere centered at $p$. The equality criteria in Gunther's volume comparison theorem (\cite{jhe}, \cite{ghl}) says that if 
$\phi (r)  = \phi_{\delta}(r)$ for $ r \leq R \leq \text{diam}(\mathbb{M})$ then the geodesic balls $B(p,R)$ and $B_{\delta}(p_0, R)$ are isometric. 
Hence we see that $\Omega $ is isometric to $ B_{\delta}(p_0, R)$.

Now suppose $M \subset \overline{\mathbb{M}}$. 
For the noncompact rank-1 symmetric spaces $(\overline{\mathbb{M}}, ds^2)$, the density function $\phi$ along 
the geodesics starting at the point $p$ is given by $\phi(r) = sinh^{kn-1}r\,cosh^{k-1}r.$ The computation is slightly different in this case.
As an illustration we give below an outline of the proof for $\mathbb{C}\mathbb{H}^n$.
For other noncompact rank-1 symmetric spaces the proof follows similarly.  

We proceed as in the earlier computation to get,
\begin{eqnarray*}
\int_M \sinh^2\,d(p,q)dm(q) & \geq & \int_A \sinh^2\,d(p,q)dm(q) \\
& = & \int_{U_p\mathbb{M}}\sinh^2\,t_q(u) \, \sec\theta (q)\,\phi(t_q(u))du\\ 
& \geq & \int_{U_p\mathbb{M}} \sinh^2\,t_q(u) \,\phi(t_q(u))du\\
& = & \int_{U_p\mathbb{M}}\sinh^{2n+1}t_q(u)\,\cosh\,t_q(u) du \\
& = & \int_{U_p\mathbb{M}}\int_0^{t_q(u)}f(r)\sinh^{2n-1}r\,\cosh\,r\, dr\, du \\
& \geq & \int_{\Omega} f(r) \, dV
\end{eqnarray*}
where $ f(r) = (2n+1)\sinh\,r\,\cosh\,r + \sinh^2r\,\tanh\,r.$ Notice that $f(r)$ is increasing for $r\,\geq\,0.$
The rest of the proof follows the same way as in the earlier case.
\end{proof}
A computation similar to the above lemma proves the following lemma.
\begin{lemma}\label{lm2}
Let $(\overline{\mathbb{M}}, ds^2)$ be a noncompact rank-1 symmetric space. Let $M$ be a closed hypersurface 
in $\overline{\mathbb{M}}$ and $\Omega$ be the bounded domain with boundary $\partial \Omega = \mathbb{M}$. Fix a point $p \in \Omega$. Then
$$
\int_M \tanh^2d(p,q)dm \geq Vol(S(p, R))\tanh^2R
$$
where $ S(p, R)$ is the geodesic sphere and $ B(p, R)$ is the geodesic ball of radius $R$ centered at $p$ in $\overline{\mathbb{M}}$ 
and $R>0$ is such that $Vol(\Omega) = Vol(B(p, R))$.\\
The equality holds if and only if $M$ is a geodesic sphere centered at $p$ of radius $R$.  
\end{lemma}
\begin{remark}
Lemma \ref{lm1} and lemma \ref{lm2} are also valid for hypersrufaces in compact rank-1 symmetric spaces.
\end{remark}
\section{Proof of Results}\label{proofs} 

Let $M$ be a closed hypersurface of $\mathbb{M}$  and 
$p \in \mathbb{M}$ be a center of mass corresponding to the mass distribution function $\frac{\sin_{\delta}r}{r}$. 
Let $X =(x_1, x_2, ..., x_{n+1})$ be the geodesic normal coordinate system centered at $p$. 
Consider the functions $g_i = f.\frac{x_i}{r}$ where $f = \sin_{\delta}r$. Then $\int_Mg_i = 0$ for $1 \leq i \leq n+1.$
Using $g_i$'s as test functions in the Rayleigh quotient, we have
\begin{eqnarray}
\lambda_1(M)\int_M\sum_{i=1}^{n+1}g_i^2 dm & \leq & \int_M\sum_{i=1}^{n+1}\parallel \nabla^M g_i \parallel^2dm \label{4.1}\\
& = & \int_M \sum_{i=1}^{n+1}g_i\Delta_Mg_idm \nonumber\\
& = & \int_Mf\Delta_Mfdm + \int_M f^2\sum_{i=1}^{n+1}f_i\Delta_Mf_idm \nonumber
\end{eqnarray}
where $f_i = \frac{x_i}{r}$. But $\sum_{i=1}^{n+1}g_i^2 = f^2$. Thus we get
\begin{equation}\label{4.2}
\lambda_1(M)\int_Mf^2 dm \leq  \int_Mf\Delta_Mfdm + \int_M f^2\sum_{i=1}^{n+1}f_i\Delta_Mf_idm.
\end{equation}
We now decompose the laplacian on $M$ as 
$$
\Delta_M =  \frac{\partial^2}{\partial \eta^2} + Tr(A)\frac{\partial}{\partial \eta} + \Delta
$$ 
where $\eta$ is the unit outward normal to $M$, $A$ is the Weingarten map of $M$ and $\Delta$ is the laplacian on $\mathbb{M}$. We do another 
decomposition of $\Delta$ along the radial geodesic starting from $p$ as 
$$ 
\Delta = - \frac{\partial^2}{\partial r^2} - Tr(A)\frac{\partial}{\partial r} + \Delta_{S(r)}
$$ 
where $\Delta_{S(r)}$ is the laplacian of $S(r)$. We also notice that
$ \sum_{i=1}^{n+1}f_i\frac{\partial f_i}{\partial \eta} = \frac{1}{2}\sum_{i=1}^{n+1}\innprod{\nabla f_i^2}{\eta} = 0 $
and $\frac{\partial}{\partial r}(f_i) = 0$. Using these we have 
\begin{center}
 \begin{eqnarray}
\sum_{i=1}^{n+1}f_i\Delta_Mf_i & = & \sum_{i=1}^{n+1}f_i\Delta f_i  + \sum_{i=1}^{n+1}f_i\frac{\partial^2f_i}{\partial \eta^2}\nonumber \\
& = & \sum_{i=1}^{n+1}f_i\Delta_{S_r}f_i - \sum_{i=1}^{n+1}\left(\frac{\partial f_i}{\partial \eta}\right)^2 \nonumber \\
& = & \sum_{i=1}^{n+1}\parallel \nabla^{S_r}f_i \parallel^2 - \sum_{i=1}^{n+1}\left(\frac{\partial f_i}{\partial \eta}\right)^2\nonumber \\
& = & \frac{1}{r^2}\sum_{i=1}^{n+1}\parallel \nabla^{S_r}x_i \parallel^2 - \sum_{i=1}^{n+1}\left(\frac{\partial f_i}{\partial \eta}\right)^2 \label{4.3}.
\end{eqnarray}
\end{center}
Thus \eqref{4.2} becomes 
\begin{eqnarray}\nonumber
\lambda_1(M)\int_Mf^2 dm & \leq & \int_M \parallel \nabla^Mf \parallel^2dm + \sum_{i=1}^{n+1}\int_M \frac{f^2\parallel \nabla^{S_r}x_i \parallel^2}{r^2}dm \\ \label{4.4}
& & -\sum_{i=1}^{n+1}\int_M f^2\left(\frac{\partial f_i}{\partial \eta}\right)^2dm.
\end{eqnarray}
Now let $M$ be a closed hypersurface in the noncompact rank-1 symmetric space $\overline{\mathbb{M}}$. 
Let $p \in \overline{\mathbb{M}}$ be a center of mass corresponding to the mass distribution function $\frac{\sinh\,r}{r}$ and let $ f= sinh\,r$. 
The same calculation as above shows that \eqref{4.4} holds in this case also. 

We now estimate $\sum\parallel \nabla^{S_r}x_i \parallel^2$. 
\\
First, consider the case $M \subset\overline{\mathbb{M}}$. It is known \cite{gs} that $\Delta_{S(r)}f_i =  \lambda_1 (S(r))f_i$,
where $\lambda_1 (S(r)) $ is the first eigen value of the laplacian  $\Delta_{S(r)}$ of the geodesic sphere $S(r) \subset\overline{\mathbb{M}}$. Thus
\begin{equation*}
 \sum_{i=1}^{kn}\frac{\parallel \nabla^{S_r}x_i \parallel^2}{r^2} = \sum_{i=1}^{kn}f_i\Delta_{S_r}f_i = \lambda_1 (S(r)).
\end{equation*}
Hence \eqref{4.3} can be written as  
\begin{equation*}	
 \sum_{i=1}^{kn}f_i\Delta_Mf_i = \lambda_1(S(r)) - \sum_{i=1}^{kn}\left(\frac{\partial f_i}{\partial \eta}\right)^2.
\end{equation*}
Substituting this in \eqref{4.4} we get for a closed hypersurface $M$ in the noncompact rank-1 symmetric space $\overline{\mathbb{M}}$,  
\begin{equation}\label{4.5}
 \lambda_1(M)\int_Mf^2 dm \leq \int_M \parallel \nabla^Mf \parallel^2dm + \int_M f^2\left(\lambda_1(S(r))-
\sum_{i=1}^{kn}\left(\frac{\partial f_i}{\partial \eta}\right)^2\right)dm.
\end{equation}

The following lemma gives an estimate of $\sum\parallel \nabla^Mx_i \parallel^2$ in the general case.
\begin{lemma}\label{l2}
Let $\delta > 0$ be given, and $\mathbb{M}$ be a complete, simply connected Riemannian manifold of dimension $(n+1)$ such that the sectional 
curvature satisfies $K_{\mathbb{M}} \leq k$ where $ k =\pm \delta^2$ or $0$. Let $M$ be a closed hypersruface and $\Omega$ be a bounded domain such that 
$\partial \Omega = M$. Fix a point $p\in \Omega $ and let $X = (x_1, x_2, ..., x_{n+1})$ be the geodesic normal
coordinate system at $p$. For $K_{\mathbb{M}} \leq \delta^2,$ 
we assume that $\Omega$ is contained in a geodesic sphere of radius $ R < \frac{\pi}{\delta}$. Then for every point $q \in M$,
$$
\sum_{i=1}^{n+1}\parallel \nabla^Mx_i(q) \parallel^2  \leq \frac{n\,d(p, q)^2}{\sin_{\delta}^2d(p, q)}.
$$
\end{lemma}
\begin{proof}
Let $q \in M$ and $(e_1, ... , e_n)$ be an orthonormal basis of $T_qM$. Then
$$
\sum_{i=1}^{n+1}\parallel \nabla^Mx_i \parallel^2  = \sum_{i=1}^{n+1}\sum_{l=1}^n \innprod{\nabla^Mx_i}{e_l}^2 = 
\sum_{i=1}^{n+1}\sum_{l=1}^n \innprod{\nabla x_i}{e_l}^2.
$$
Let $e_l = d(exp_p)\bar{e_l}. $
Note that $\innprod{\nabla x_i}{e_l} = e_l(x_i) = \bar{e_l}(x_i\circ exp_p)$ is the $i$-th component of 
$\bar{e_l}$ in the geodesic normal coordinate at $p$. Thus 
$$
\sum_{i=1}^{n+1}\innprod{\nabla x_i}{e_l}^2 = \parallel \bar{e_l} \parallel^2.
$$
Consider a unit speed geodesic $\gamma$ in $\mathbb{M}$ such that $\gamma(0) = p$ and $\gamma(r) = q$. Let 
$J_l$ be the Jacobi field along $\gamma$ such that 
$J_l(0) = 0 $ and $J_l^{\prime}(0) = \bar{e_l}$. Then $e_l = d(exp_p)\bar{e_l} = \frac{1}{r}J_l(r).$ 
Let $\mathbb{M}(k)$ be the space form with constant curvature $k = \pm \delta^2$ or $0$. Fix a point $p_0 \in \mathbb{M}(k)$
and a unit speed geodesic $\bar{\gamma}$ such that $\bar{\gamma}(0) =  p_0$.
Let $\bar{u}$ be a unit vector at $p_0$ and $E(t)$ be the vector 
field obtained by parallel translating $\bar{u}$ along $\bar{\gamma}(t)$. Consider the Jacobi field 
$$
J_{\delta}(t) = \sin_{\delta}t\parallel J_l^{\prime}(0) \parallel E(t)
$$ 
along $ \bar{\gamma}$. By the Rauch comparison theorem 
$$
\parallel J_{\delta}(t) \parallel \leq \parallel J_l(t) \parallel \ \ \ \  \text{for} \ \ \  
\begin{cases}
 0 \leq t < \frac{\pi}{\delta}  & if \  K \leq \delta^2 \\
 t \geq 0 & if \  K \leq -\delta^2 \, \, \text{or} \, \, 0. 
\end{cases}
$$
Hence $\parallel e_l \parallel^2  = \frac{1}{r^2}\parallel J_l(r) \parallel^2 \geq \frac{1}{r^2} \parallel J_{\delta}(r) \parallel^2 
 = \frac{\sin_{\delta}^2r}{r^2}\parallel J_l^{\prime}(0) \parallel^2,$ which implies
$$
\parallel \bar{e_l} \parallel^2 = \parallel J_l^{\prime}(0) \parallel^2 \leq \frac{r^2}{\sin_{\delta}^2r}.
$$
Thus we get 
$$
\sum_{i=1}^{n+1}\parallel \nabla^Mx_i(q) \parallel^2  = \sum_{i=1}^{n+1}\sum_{l=1}^n \innprod{\nabla x_i}{e_l}^2 
= \sum_{l=1}^n \parallel \bar{e_l} \parallel^2 \leq \frac{n\,d(p, q)^2}{\sin_{\delta}^2d(p, q)}.
$$  
\end{proof}
\begin{proof}[Proof of theorem \ref{thm1}]
 We first consider the case $0 \leq K_{\mathbb{M}} \leq \delta^2$. Recall the inequality \eqref{4.4} 
\begin{eqnarray*}
\lambda_1(M)\int_Mf^2 dm & \leq & \int_M \parallel \nabla^Mf \parallel^2dm + \int_M \frac{f^2}{r^2}\sum_{i=1}^{n+1}\parallel \nabla^{S_r}x_i \parallel^2dm \\
& & -\int_Mf^2 \sum_{i=1}^{n+1}\left(\frac{\partial f_i}{\partial \eta}\right)^2dm.
\end{eqnarray*}
Applying lemma \ref{l2} for the geodesic sphere $S_r$, we get $\sum_{i=1}^{n+1}\parallel \nabla^{S_r}x_i \parallel^2 \leq \frac{n\,r^2}{\sin_{\delta}^2r}. $
Substituting this in the above inequality, we get 
\begin{eqnarray}\nonumber
\lambda_1(M)\int_M\sin_{\delta}^2d(p, q) dm(q) & \leq & n\,Vol(M) +\int_M \parallel \nabla^M\sin_{\delta}d(p, q) \parallel^2dm(q)   
\\ 
& &-\int_M\sin_{\delta}^2d(p,q)\sum_{i=1}^{n+1}\left(\frac{\partial f_i}{\partial \eta}(q)\right)^2dm(q)\label{4.6}.
\end{eqnarray}
We now estimate $\sum_{i=1}^{n+1}\left(\frac{\partial f_i}{\partial \eta}(q)\right)^2$ for $q \in M$.

For $q \in M,$ let $\eta(q) = a\partial r(q) + bv$ where $v\in T_q\mathbb{M},\parallel v \parallel = 1, \innprod{\partial r(q)}{v} = 0$.
Note that $b^2 = \parallel \nabla^Mr(q) \parallel^2$ and $\innprod{\nabla f_i}{\eta}(q) = 
a \innprod{\nabla f_i}{\partial r}(q) + b\innprod{\nabla f_i(q)}{v}$. But $\nabla f_i = \frac{\nabla x_i}{r} - \frac{x_i}{r^2}\partial r$. Thus
$\innprod{\nabla f_i}{\eta}(q) = \frac{b}{r}\innprod{\nabla x_i(q)}{v}$. Let $ v = d(exp_p)\bar{v} = \frac{1}{r}J(r)$ where
$J$ is the Jacobi field along the unit speed geodesic $\gamma_q$ which joins $p$ and $q$ such that 
$J(0) = 0 $ and $J^{\prime}(0) = \bar{v}$. Note that$\innprod{\nabla x_i(q)}{v} = v(x_i) = \bar{v}(x_i\circ exp_p)$ is the $i$-th component of 
$\bar{v}$ in the geodesic normal coordinate at $p$. Thus 
\begin{equation*}
\sum_{i=1}^{n+1}\innprod{\nabla x_i(q)}{v}^2 = \parallel \bar{v} \parallel^2. 
\end{equation*}
Consider $\mathbb{R}^{n+1}$ and fix a point $p_0 \in \mathbb{R}^{n+1}$. Let $\alpha$ be a unit speed 
geodesic such that $\alpha(0) =  p_0$. Let $\bar{u}$ be a unit vector at $p_0$ and $W(t)$ be the vector 
field obtained by parallel translating $\bar{u}$ along $\alpha(t)$. Consider the Jacobi field along $J_0(t)$ 
along $\alpha(t)$ given by
$$J_0(t) = \parallel J^{\prime}(0) \parallel \,t W(t).
$$
By Rauch comparison theorem we see that 
$
\parallel J(t) \parallel \leq \parallel J_0(t) \parallel \, \text{for} \, t \, < inj(\mathbb{M}). 
$
Thus 
$$
1 = \parallel v \parallel^2 = \frac{1}{r^2}\parallel J(r) \parallel^2 \leq \frac{1}{r^2} \parallel J^{\prime}(0) \parallel^2 r^2
= \parallel \bar{v} \parallel^2
$$
which gives
\begin{equation*}
\sum_{i=1}^{n+1}\innprod{\nabla f_i(q)}{\eta(q)}^2 =\frac{b^2}{r^2} \sum_{i=1}^{n+1}\innprod{\nabla x_i(q)}{v}^2 = \frac{b^2}{r^2}\parallel \bar{v} \parallel^2 
\geq \frac{b^2}{r^2}.
\end{equation*}
Substituting this, \eqref{4.6} becomes  
\begin{eqnarray*}
\lambda_1(M)\int_M\sin_{\delta}^2d(p, q) dm(q) & \leq & n\,Vol(M) + \int_M \parallel\nabla^M\sin_{\delta}d(p, q)\parallel^2dm(q) \\
&  & -\int_M\sin_{\delta}^2d(p, q)\frac{\parallel \nabla^Mr \parallel^2}{r^2}dm(q).
\end{eqnarray*}
Also we have $\parallel \nabla^M\sin_{\delta}r \parallel^2 = \parallel \nabla^Mr \parallel^2\cos_{\delta}^2r $. Hence 
\begin{equation*}
 \lambda_1(M)\int_M\sin_{\delta}^2d(p, q) dm \leq n\,Vol(M) + \int_M  \parallel \nabla^Mr \parallel^2 \left(\cos_{\delta}^2r - 
\frac{\sin_{\delta}^2r}{r^2}\right)dm(q).
\end{equation*}
As $\tan{\delta}\,r = \frac{\sin{\delta}r}{\cos{\delta}r} \geq \delta\,r,$ for $0 \leq r < \frac{\pi}{\delta},$ we get
\begin{equation*}
 \lambda_1(M)\int_M\sin_{\delta}^2d(p, q) dm \leq n\,Vol(M). 
\end{equation*}
Now by lemma \ref{lm1}, $\int_M\sin_{\delta}^2d(p, q) dm \geq Vol(S_{\delta}(R))\sin_{\delta}^2R $. Substituting this in the above inequality we get, 
$$
 \lambda_1(M) \leq \frac{n}{\sin_{\delta}^2R}\left(\frac{Vol(M)}{Vol(S_{\delta}(R))}\right).
$$
Since $\frac{n}{\sin_{\delta}^2R} = \lambda_1(S_{\delta}(R))$, we have 
\begin{equation}\label{4.7}
 \lambda_1(M) \leq \lambda_1(S_{\delta}(R))\left(\frac{Vol(M)}{Vol(S_{\delta}(R))}\right)
\end{equation}
which is the desired inequality.

As $\tan\delta \,r > \delta r$ for $0 < r < \frac{\pi}{\delta},$ the equality holds if and only if
$\parallel \nabla^Mr \parallel = 0$ and the equality in lemma(\ref{lm1})
holds. This shows that equality in \eqref{4.7} holds if and only if
$M$ is geodesic sphere in $\mathbb{M}$ and $\Omega$ is isometric to $B_{\delta}(R)$.

Next consider the case $K_{\mathbb{M}} \leq 0$. Notice that $g_i = x_i$ for $i=1, ..., n+1$. Hence \eqref{4.1} can be written as 
\begin{equation}
\lambda_1(M)\int_M\sum_{i=1}^{n+1}x_i^2 dm  \leq  \int_M\sum_{i=1}^{n+1}\parallel \nabla^M x_i \parallel^2dm \label{4.8}.
\end{equation}
By lemma \ref{l2} we get 
\begin{equation*}
 \sum_{i=1}^{n+1}\parallel \nabla^M x_i \parallel^2 \leq n.
\end{equation*}
Also by lemma \ref{lm1}, we have
$$\int_M\sum_{i=1}^{n+1}x_i^2 dm = \int_Mr^2 dm \geq R^2\,Vol(S(R)).
$$
Substituting these into \eqref{4.8} and using the fact that $\frac{n}{R^2} = \lambda_1(S(R))$, we get the desired inequality
$$
\lambda_1(M) \leq \lambda_1(S(R))\left(\frac{Vol(M)}{Vol(S(R))}\right).
$$
Further, the equality condition follows from the equality condition in the lemma \ref{lm1}. 
\end{proof}
\begin{proof}[Proof of theorem \ref{thm2}]
The proof in this case is similar to the proof of theorem \ref{thm1} except that we do not have a similar estimate of 
\begin{equation*}
\int_M \parallel \nabla^M\sin_{\delta}d(p, q) \parallel^2dm(q)   
-\int_M\sin_{\delta}^2d(p,q)\sum_{i=1}^{n+1}\left(\frac{\partial f_i}{\partial \eta}(q)\right)^2dm(q).
\end{equation*}
Thus from \eqref{4.6}, 
\begin{equation*}
\lambda_1(M)\int_M\sin_{\delta}^2d(p, q) dm(q)  \leq  n\,Vol(M) +\int_M \parallel \nabla^M\sin_{\delta}d(p, q) \parallel^2dm(q).
\end{equation*}
By lemma \ref{lm1}, we get the desired inequality 
\begin{equation}\label{4.9}
\frac{\lambda_1(M)}{\lambda_1(S_{\delta}(R))} \leq \frac{Vol(M)}{Vol(S_{\delta}(R))}\ +
\frac{1}{n\,Vol(S_{\delta}(R))}\int_M \parallel \nabla^M\sin_{\delta}r\parallel^2. 
\end{equation}

The equality holds if and only if the equality in lemma \ref{lm1} holds and 
$\frac{\partial f_i}{\partial \eta}(q) = 0$ for all $ i = 1, ... , n+1 $ and for all points $q \in M$. The later is true if and only if 
$\eta(q) = \partial r(q)$ for all points $ q \, \in M,$ which implies that $M$ is a geodesic sphere. Thus the equality in \eqref{4.9} holds if and 
only if $M$ is a geodesic sphere and $\Omega$ is isometric to $B_{\delta}(R)$.  
\end{proof}
\begin{proof}[Proof of theorem \ref{thm3}]
We recall the inequality \eqref{4.5},
\begin{eqnarray*}
 \lambda_1(M)\int_M\sinh^2rdm & \leq & \int_M \parallel \nabla^M\sinh\,r\parallel^2dm + \int_M \sinh^2r\,\lambda_1(S(r))\,dm \\
 & & -\int_M\sum_{i=1}^{kn}\left(\frac{\partial f_i}{\partial \eta}\right)^2\,dm.
\end{eqnarray*}
Substituting for $\lambda_1(S(r))$ in the above inequality we get
\begin{eqnarray*}
 \lambda_1(M)\int_Mf^2 dm & \leq & (kn-1)Vol(M) - (k-1)\int_M tanh^2rdm \\
& & + \int_M \parallel \nabla^M\sinh\,r \parallel^2dm. 
\end{eqnarray*}
By lemma \ref{lm1} and lemma \ref{lm2}, we get 
 \begin{align*}
 \lambda_1(M)Vol(S(R))\sinh^2R  & \, \leq \, (kn-1)Vol(M) -(k-1)\tanh^2R\,Vol(S(R))\\
  & \quad \ \, + \int_M \parallel \nabla^M\sinh\,r \parallel^2\,dm. 
\vspace{-.3cm} 
\end{align*}
Now suppose that $k = 1$. Then above inequality reduces to
\begin{align*}
 \lambda_1(M)Vol(S(R))\sinh^2R  & \, \leq \, (n-1)Vol(M) + \int_M \parallel \nabla^M\sinh\,r \parallel^2\,dm. 
\end{align*}
Using the fact that $\lambda_1(S(r)) = \frac{n-1}{\sinh^2r}$ for all $r>0$, we get the required result 
\begin{equation}
\frac{\lambda_1(M)}{\lambda_1(S(R))} \leq \frac{Vol(M)}{Vol(S(R))} + \frac{1}{(n-1)Vol(S(R))}\int_M \parallel \nabla^M\sinh\,r 
\parallel^2 \label{hyp}
\end{equation}
for hypersurfaces in $\mathbb{H}^n$.

When $k > 1 $, we get
\begin{eqnarray}\nonumber
  \lambda_1(M) & \leq & \left(\frac{kn-1}{\sinh^2R} - \frac{k-1}{\cosh^2R}\right)\frac{Vol(M)}{Vol(S(R))}\\ \nonumber
& & + \frac{1}{Vol(S(R))}\left(\frac{k-1}{\cosh^2R}Vol(M) + \frac{1}{\sinh^2R}\int_M\! \parallel \!\nabla^M\sinh\,r \!\parallel^2\right)\\ \nonumber
 & = & \lambda_1(S(R))\left(\frac{Vol(M)}{Vol(S(R))}\right) + \frac{k-1}{\cosh^2R}\left(\frac{Vol(M)}{Vol(S(R))}\right) \\
& & + \ \frac{1}{\sinh^2R\,Vol(S(R))}\!\int_M\! \parallel \!\nabla^M\sinh\,r \!\parallel^2. \label{nonhyp}
\end{eqnarray}

The equality in \eqref{hyp} and in \eqref{nonhyp} follows from the equality criterion in lemma \ref{lm1} and lemma \ref{lm2} and the fact that  
$\frac{\partial f_i}{\partial \eta}(q) = 0$ for all $ i = 1, ... , kn $ and for all points $q \in M$ happens if and only if $M$ is a geodesic sphere.
\end{proof}
\begin{remark}
 In the case of $\mathbb{H}^n$, a Jacobi field computation gives 
$$
\sum_{i=1}^{kn}\left(\frac{\partial f_i}{\partial \eta}\right)^2 = \frac{1}{\sinh^2r}\parallel \nabla^Mr \parallel^2.
$$
This implies that
$$
\int_M \parallel \nabla^M\sinh\,r \parallel^2  - \int_M f^2 \sum_{i=1}^{kn}\left(\frac{\partial f_i}{\partial \eta}\right)^2
= \int_M \parallel \nabla^Mcosh\,r \parallel^2.
$$
Thus \eqref{hyp} becomes
$$
\frac{\lambda_1(M)}{\lambda_1(S(R))} \leq \frac{Vol(M)}{Vol(S(R))} + \frac{1}{(n-1)Vol(S(R))}\int_M \parallel \nabla^Mcosh\,r \parallel^2.
$$
\end{remark}


\begin{thebibliography}{99}
\bibitem{bw}
D.~Bleecker and J.~Weiner.:
\newblock Extrinsic bounds on $\lambda_1$ of $\Delta$ on a compact manifold.
\newblock {\em Comment. Math. Helv.}, 51:601--609, 1976.

\bibitem{ic}
I. Chavel.:
\newblock Eigenvalues in Riemannian Geometry,
\newblock Academic Press, Inc. 1984.

\bibitem{dc} 
M.P do Carmo.:
\newblock Riemannian Geometry, 
\newblock Birkh$\ddot{a}$user Boston 1992.

\bibitem{jhe}
J. H. Eschenburg.:
\newblock Comparison theorems and hypersurfaces.
\newblock {\em Manuscripta Math.}, 59:295--323, 1987.

\bibitem{gjf}
Grosjean~J F.:
\newblock Upper bounds for the first eigenvalue of the laplacian on compact
  submanifolds.
\newblock {\em Pacific. J. Math.}, 206:93--112, 2002.

\bibitem{ghl}
S. Gallot, D. Hulin and J. Lafontaine.:
\newblock Riemannian Geometry. 
\newblock Third edition. Springer, 2004.

\bibitem{eh}
E. Heintze.:
\newblock Extrinsic upperbounds for $\lambda_1$.
\newblock {\em Math. Ann.}, 280:389--402, 1988.

\bibitem{mtw}
J. E. Marsden, A. J. Tromba, A. Weinstein.:
\newblock Basic Multivariable Calculus.
\newblock Springer India Pvt. Ltd, 2009.

\bibitem{r}
R.Reilly.:
\newblock On the first eigenvalue of the laplacian for compact submanifold of
  euclidean space.
\newblock {\em Comment. Math. Helv.}, 52:525--533, 1977.

\bibitem{gs}
G.~Santhanam.:
\newblock A sharp upperbound for the first eigenvalue value of the laplacian of
  compact hypersurfaces in rank-1 symmetric spaces.
\newblock {\em Proc. Indian Acad. Sci. (Math. Sci)}, 117, No. 3:307--315,
  August 2007.

\bibitem{gs1}
G. Santhanam.:
\newblock Isoperimetric upper bounds for the first eigenvalue.
\newblock {\em Preprint.}

\bibitem{tf} 
G. B. Thomas, JR and R. L. Finney.:
\newblock Calculus and Analytical Geometry.
\newblock Addison-Wesley, $9^{th}$edition.
 
\end{thebibliography}
\end{document}